%% file: ann_min_rep.tex
\documentclass[a4paper,reqno,12pt]{amsart}   
\usepackage{newtxtext} 
\usepackage{amsmath}
\usepackage[bigdelims]{newtxmath}
\usepackage[T1]{fontenc}
\usepackage{textcomp}

\usepackage{tikz}
\usepackage{tikz-cd}
\usepackage{mathrsfs}
\usepackage{ytableau}

\ytableausetup{boxsize=5pt}

\tikzcdset{
	arrow style=tikz,
	diagrams={>=stealth}
}

\input env.tex

\input ncpfaff_def.tex

\DeclareMathOperator{\impart}{Im}

\DeclareMathOperator{\spanned}{span}

\DeclareMathOperator{\gr}{gr}

\DeclareMathOperator{\Ann}{Ann}

\newcommand{\rmO}{\mathrm O}
\newcommand{\gnat}{{\mathfrak o}_n}
\newcommand{\nat}[1]{\ensuremath{{#1}^{\gnat}}}

\newcommand{\ai}{\ensuremath{\sqrt{-1}\,}}

\newcommand{\abs}[1]{\ensuremath{\left|\, #1 \, \right|}}

\newcommand{\PD}{\ensuremath{\mathscr{P}\!\mathscr{D}}}

\newcommand{\Eu}{\ensuremath{E}}

\newcommand{\Harm}{\ensuremath{\mathscr H}}
\newcommand{\ourE}{\ensuremath{\mathscr E}}
\newcommand{\Module}[2]{\ensuremath{\mathit M}^{#1}(#2)}
\newcommand{\Anni}[2]{\ensuremath{\mathscr I}^{#1}(#2)}

\newcommand{\Cspan}[1]{\mbox{\C -}\! \spanned \left\{ {#1} \right\} }

\newcommand{\mwedge}[1]{\ensuremath{\mbox{$\bigwedge^{\!#1}$}}}
\newcommand{\sfn}{\ensuremath{\mathsf n}}

\begin{document}

\title[Annihilator of $(\g,K)$-modules of $\rmO(p,q)$]
{Annihilator of $(\g,K)$-modules of $\rmO(p,q)$ associated to the finite-dimensional representation of $\mathfrak{sl}_2$}

\author{Takashi Hashimoto}

\address{
  Center for Data Science Education, 
  Tottori University, 
  4-101, Koyama-Minami, Tottori, 680-8550, Japan
}
\email{thashi@tottori-u.ac.jp}
\date{\today}

\begin{abstract}
Let $\g$ denote the complexified Lie algebra of $G=\rmO(p,q)$ and $K$ a maximal compact subgroup of $G$.
In the previous paper, we constructed $(\g,K)$-modules associated to the finite-dimensional representation of $\mathfraksl_2$ of dimension $m+1$, which we denote by $\Module{+}{m}$ and $\Module{-}{m}$.
The aim of this paper is to show that the annihilator of $\Module{\pm}{m}$ is the Joseph ideal if and only if $m=0$.
We shall see that an element of the symmetric of square $S^{2}(\g)$ that is given in terms of the Casimir elements of $\g$ and the complexified Lie algebra of $K$ plays a critical role in the proof of the main result.
\end{abstract}

\maketitle

\section{Introduction}

For a complex simple Lie algebra $\g$, Joseph constructed the so-called Joseph ideal in \cite{Joseph76}, which is a completely prime two-sided primitive ideal in the universal enveloping algebra $U(\g)$ of $\g$ and has the associated variety of minimal dimension. 
He also proved that it is unique if $\g$ is not of type $A$ (see also \cite{GS04}).
An irreducible unitary representation of a simple Lie group is called a minimal representation if its annihilator in $U(\g)$ is equal to the Joseph ideal. 

The best known example of the minimal representation is the oscillator (or Segal-Shale-Weil) representation of the metaplectic group, the double cover of the symplectic group $\Sp(n,\R)$.
It was shown in \cite{harm_oscil} that the canonical quantization of the moment map on real symplectic vector spaces gives rise to the underlying $(\g,K)$-module of the oscillator representation of real reductive Lie groups $G$,
where $\g$ denotes the complexified Lie algebra of $G$ and $K$ a maximal compact subgroup of $G$.

In the case of $G=\rmO(p,q)$, the indefinite orthogonal group, applying this method to the symplectic vector space $( (\C^{p+q})_{\R}, \omega)$ with $\omega(z,w)=\impart(z^{*} I_{p,q} w)$, on which $G$ acts via matrix multiplication in a Hamiltonian way with moment map (see below for details), 
and making use of the fact that $(\rmO(p,q), \SL{2}(\R))$ is a dual pair in $\Sp(p+q,\R)$, 
the author constructed $(\g,K)$-modules $\Module{+}{m}$ and $\Module{-}{m}$ of $\rmO(p,q)$ associated to the finite-dimensional representation of $\mathfraksl_{2}$ of dimension $m+1$ for any non-negative integer $m$ in \cite{min_rep}.
Note that $\Module{+}{0}=\Module{-}{0}$ by definition and $\Module{+}{m}$ is naturally isomorphic to $\Module{-}{m}$ in fact.
Moreover, under the condition
\begin{equation}
\label{e:assumpntion_on_pq in intro}
	p \geqsl 2,\; q \geqsl 2,\quad p+q \in 2\N \quad \text{and}\quad  m+3 \leqsl \frac{p+q}2,
\end{equation}
he obtained the $K$-type formula of $\Module{\pm}{m}$, thereby showed that they are irreducible and that $\Module{\pm}{0}$ is the underlying $(\g,K)$-module of the minimal representation of $\rmO(p,q)$.
Although it is well known that the Gelfand-Kirillov dimension of the minimal representation of $\rmO(p,q)$ is equal to $p+q-3$ (see, e.g., \cite{Kobayashi_Orsted_03-2,ZH97}),
we remark that the $K$-type formula of $\Module{\pm}{m}$ mentioned above implies that they all have the Gelfand-Kirillov dimension equal to $p+q-3$, whereas they have the Bernstein degree equal to $4(m+1)\binom{p+q-4}{p-2}$. 

The original object of this project is to clarify the relation between our $(\g,K)$-modules $\Module{\pm}{m}$ and the minimal representation of $\rmO(p,q)$ more directly.
In fact, we shall see that the annihilator in the universal enveloping algebra $U(\g)$ of $\Module{\pm}{m}$, denoted by $\Anni{\pm}{m}$ in this paper, is the Joseph ideal if and only if $m=0$, which is the main result of this paper (Theorem \ref{t:main result}).
We use a criterion due to Garfinkle \cite{Garfinkle82} (see also \cite{GS04}) to determine whether $\Anni{\pm}{m}$ is the Joseph ideal or not.
We will see that an element of the symmetric square $S^2(\g)$, which we denote by $\Xi$ below, plays a critical role in the proof of our main result.
We remark that $\Xi$ is given in terms of the Casimir elements of $\g$ and $\mathfrak k$, where $\mathfrak k$ is the complexified Lie algebra of the maximal compact subgroup $K$.

The rest of the paper is organized as follows.
In Section 2 we briefly review the construction of $(\g,K)$-modules $\Module{\pm}{m}$ of $\rmO(p,q)$ mentioned above as well as their properties needed to prove our main results.
In Section 3 we prove our main result using the criterion due to Garfinkle.
In order to keep the present paper self-contained, we include some standard facts about irreducible decomposition of $S^2({\mathfrak o}_n)$ in the Appendix.

\subsection*{Notation}
Let $\N$ denote the set of non-negative integers.
For positive integers $p$ and $q$, we set $[p]:=\{1,2,\dots,p\}$ and $p+[q]:=\{p+1,p+2,\dots,p+q\}$.

%
%
\section{$(\g,K)$-modules of $\rmO(p,q)$ associated to the finite-dimensional representation of $\mathfrak{sl}_{2}$}
\label{s:sec1}

Throughout this paper, we denote by $G$ the indefinite orthogonal group $\rmO(p,q)$ realized by
\[
  \rmO(p,q) = \left\{ g \in \GL{n}(\R) ; \tp{g} I_{p,q} g = I_{p,q} \right\},
\]
where $n=p+q$ and $I_{p,q} = \left[ \begin{smallmatrix} 1_p & \\  & -1_q \end{smallmatrix} \right]$.
Let $\g={\mathfrak o}(p,q) \otimes \C$, the complexified Lie algebra of $G$.
Thus,
\[
	\g = \left\{ X \in \gl_{n}(\C) ; \tp{X} I_{p,q} + I_{p,q} X = O \right\}.
\]
Needless to say, our Lie algebra $\g$ is $\C$-isomorphic to the Lie algebra $\gnat$ consisting of all complex $n \times n$ alternating matrices:
\[
	\gnat = \left\{ X \in \gl_{n}(\C) ; \tp{X} + X = O \right\}.
\]

Let $K$ be a maximal compact subgroup of $G$ given by 
\[
  K = \left\{ \begin{bmatrix} a & 0 \\ 0 & d \end{bmatrix} \in G \, ; a \in \rmO(p), d \in \rmO(q) \right\}
    \simeq \rmO(p) \times \rmO(q).
\]
Denote the Lie algebra of $K$ and its complexification by $\mathfrak k_{0}$ and $\mathfrak k$ respectively.
Then $\mathfrak k \simeq \mathfrak o_{p} \oplus \mathfrak o_{q}$.

Let us fix a nondegenerate invariant bilinear form $B$ defined both on $\g$ and on $\gnat$ as follows:
\begin{equation}
	B(X,Y)=\frac12 \trace{XY}.
\end{equation}
Let $\{ X_{i,j}\}_{i,j \in [n]}$ be the generators of $\g$ given by
\begin{align}
\label{e:basis4g}
  X_{i,j} &= \epsilon_{j} E_{i,j} - \epsilon_{i} E_{j,i} 
\end{align}
for $i,j=1,2,\dots,n$, where $E_{i,j}$ stands for the matrix unit and
\begin{equation}
	\epsilon_{i} = \begin{cases} 1 & \text{if} \; i \in [p], \\ -1 &\text{if} \; i \in p+[q]. \end{cases}
\end{equation}
Then one sees $X_{j,i}=-X_{i,j}$ and
\[
	\mathfrak k=\bigoplus_{\substack{i,j \in [p] \\ i<j}} \C X_{i,j} \oplus \bigoplus_{\substack{i,j \in p+[q] \\ i<j}} \C X_{i,j}.
\]
In addition, if one sets
\[
	\mathfrak p=\bigoplus_{\substack{i \in [p] \\ j \in p+[q]}} \C X_{i,j},
\]
then $\g=\mathfrak k \oplus \mathfrak p$ is the complexified Cartan decomposition.
Finally, the dual generators $\{X_{i,j}^{\vee}\}_{i,j}$ to $\{X_{i,j}\}_{i,j}$ with respect to $B$, i.e., those satisfying $B(X_{i,j}, X_{k,l}^{\vee})=\delta_{i,k}\delta_{j,l}$, are given by
\begin{equation}
\label{e:dual_elements}
	X_{i,j}^{\vee} = \begin{cases} -X_{i,j} &\text{if\;} i,j \in [p], \text{\;or if\;} i,j \in p+[q], \\ X_{i,j} &\text{otherwise}. \end{cases}
\end{equation}

Let $V$ denote the vector space $\R^{p+q}$ equipped with a quadratic form $\eta$ defined by
\[
	\eta(v) = \tp{v} I_{p,q} v	\quad (v \in \R^{p+q}).
\]
Then $G$ acts on $V$ by matrix multiplication, leaving $\eta$ invariant.
It is well known that the space $\C[V]$ of polynomial functions on $V$ decomposes as follows (see e.g.~\cite{GW98}): 
for $v=\tp{(x_{1},\dots,x_{p},y_{1},\dots,y_{q})} \in V$, one sees that
\begin{align}
 \C[V]  & \simeq \C[x_1,\dots,x_p] \otimes \C[y_1,\dots,y_q] 
           \notag \\
      & \simeq \bigoplus_{k=0}^{\infty} \left( \C [ r_x^{2} ]  \otimes \Harm^k(\R^p) \right) 
                  \otimes \bigoplus_{l=0}^{\infty} \left( \C [ r_y^{2} ] \otimes \Harm^l(\R^q) \right) 
           \notag \\
      & \simeq  \bigoplus_{k,l=0}^{\infty}  \Harm^k(\R^p) \otimes \Harm^l(\R^q) \otimes \C [ r_x^2, r_y^2 ],
\label{e:decompositionC[V]}
\end{align}
where $r_{x}^{2}=x_{1}^{2}+\cdots+x_{p}^{2},\,r_{y}^{2}=y_{1}^{2}+\cdots+y_{q}^{2}$ and $\Harm^k(\R^{p})$ denotes the space of homogeneous harmonic polynomials on $\R^{p}$ of degree $k$ and so on.

Taking \eqref{e:decompositionC[V]} into account, let us introduce our space $\ourE$ of functions on $V$ by
\begin{equation}
\label{e:ambient_function_space}
  \ourE := \bigoplus_{k,l=0}^\infty {\Harm}^k(\R^p) \otimes {\Harm}^l(\R^q) \otimes \C \llbracket r_x^2,r_y^2 \rrbracket,
\end{equation}
where $\C \llbracket r_{x}^{2},r_{y}^{2} \rrbracket$ denotes the ring of formal power series in $r_{x}^{2}$ and $r_{y}^{2}$, and the direct sum is algebraic;
we assume that the radial part should be a formal power series rather than a polynomial so that we can find a highest weight vector or a lowest weight vector with respect to the $\mathfraksl_{2}$-action \eqref{e:TDS} given below.

%
%
\begin{definition}
For an element $f \in \ourE$ of the form $f=h_{1}(x) h_{2}(y) \phi(r_{x}^{2},r_{y}^{2})$ with $h_{1} \in \Harm^{k}(\R^{p}),\;h_{2} \in \Harm^{l}(\R^{q})$ and $\phi(r_{x}^{2},r_{y}^{2}) \in \C \llbracket r_{x}^{2},r_{y}^{2} \rrbracket$,
we set $\kappa_{+}=\kappa_{+}(f)=k+{p}/2,\; \kappa_{-}=\kappa_{-}(f)=l+{q}/2$, and
\begin{equation}
	\kappa(f) = (\kappa_{+},\kappa_{-}) = \left( k+\frac{p}2, l+\frac{q}2 \right),
\end{equation}
which we call the \emph{$K$-type} of $f=h_{1}(x) h_{2}(y) \phi(r_{x}^{2},r_{y}^{2})$ in this paper by abuse.

Furthermore, for non-negative integers $k,l \geqsl 0$ given, 
we say that an element of $\ourE$ is \emph{$K$-homogeneous with $K$-type $(\kappa_{+},\kappa_{-})=(k+p/2,l+q/2)$} if it is a linear combination of elements whose $K$-types are all equal to $(\kappa_{+},\kappa_{-})$.
\end{definition}

Denoting the ring of polynomial coefficient differential operators on $V$ by $\PD(V)$, let us define a representation $\pi: U(\g) \to \PD(V)$ by
\begin{equation}
\label{e:varpi}
\pi(X_{i,j}) = \begin{cases}
					x_i \pd_{x_j} - x_j \pd_{x_i}  				&\text{if} \quad i,j \in [p], \\
					- y_{i'} \pd_{y_{j'}} + y_{j'} \pd_{y_{i'}}  		&\text{if} \quad i,j \in p+[q], \\
					-\ai ( x_i y_{j'} + \pd_{x_i} \pd_{y_{j'}})   		&\text{if} \quad i \in [p], j \in p+[q], \\
					\ai ( x_j y_{i'} + \pd_{x_j} \pd_{y_{i'}} )		&\text{if} \quad i \in p+[q], j \in [p],
				\end{cases}
\end{equation}
where we set $i':=i-p, \; i \in p+[q]$, for short.
It is easy to verify that $\pi$ indeed defines a representation of $U(\g)$.
It is also clear that the action of $\mathfrak{k}_{0}$ lifts to the one of $K$, i.e., if $X \in \mathfrak k_{0}$ then
\begin{equation}
	\big( \pi(X)f \big) (v) = \left. \frac{\rmd}{\rmd t} \right|_{t=0} f \big( \exp (-t X) v \big) \quad (f \in \ourE, v \in V).
\end{equation}

\begin{remark}
\label{r:canonical quantization}
It was shown in \cite{min_rep} that if one regards $V$ as a Lagrangian subspace of the symplectic vector space $( (\C^{p+q})_{\R},\omega )$ given by
\[
	V=\<e_{1},e_{2},\dots,e_{p},\ai e_{p+1}, \ai e_{p+2},\dots, \ai e_{p+q} \>_{\R},
\]
where $e_{i}=\tp{(0,\dots,\overset{i\text{-th}}{1},\dots,0)} \in \C^{p+q}$ and
\[
	\omega(z,w)=\impart ({z}^{*} I_{p,q} w) \quad (z,w \in \C^{p+q}),
\]
then one obtains the representation $\pi$ given in \eqref{e:varpi} via the canonical quantization of the moment map on $( (\C^{p+q})_{\R},\omega )$.
\end{remark}

Define elements $H,X^{+},X^{-}$ of $\PD(V)$ by
\begin{equation}
\label{e:TDS}
 H = -\Eu_x - \frac{p}2 + \Eu_y +\frac{q}2, \quad
 X^{+} = -\frac12 ( \Delta_x + r_y^2 ), \quad
 X^{-} = \frac12 ( r_x^2 + \Delta_y ),
\end{equation}
where 
\begin{equation}
\label{e:TDS_in_x_and_y}
\begin{aligned}
 \Eu_x &= \sum_{i \in [p]} x_i \pd_{x_i},  & \quad   r_x^2 &= \sum_{i \in [p]} x_i^2,  & \quad  \Delta_x &= \sum_{i \in [p]} \pd_{x_i}^2,  \\
 \Eu_y &= \sum_{j \in [q]} y_j \pd_{y_j},  & \quad   r_y^2 &= \sum_{j \in [q]} y_j^2,  & \quad  \Delta_y &= \sum_{j \in [q]} \pd_{y_j}^2.
\end{aligned}
\end{equation}
Then they satisfy that
\begin{equation}
 [H, X^{+}] = 2 X^{+}, \quad [H, X^{-}] = -2 X^{-}, \quad [X^{+}, X^{-}] = H.
\notag
\end{equation}
Namely, $\g':=\Cspan{ H, X^{+}, X^{-} }$ forms a Lie subalgebra of $\PD(V)$ isomorphic to $\mathfraksl_2$.
Moreover, $\g'$ forms the commutant of $\g$ in $\PD(V)$:
\[
	[Y,Z]=0 \quad \text{for all}\;Y \in \pi(\g) \; \text{and}\; Z \in \g'.
\]

In view of the fact that $(\rmO(p,q),\mathrm{SL}(2,\R))$ is a dual pair in $\Sp(n,\R)$,  
we introduce $(\g,K)$-modules of $\rmO(p,q)$ associated to the finite-dimensional representation of $\mathfraksl_{2}$ as follows.
%
%
\begin{definition}[\cite{min_rep}]
\label{d:Module^{pm}_{m}}
Given $m \in \N$, we define $(\g,K)$-modules $\Module{\pm}{m}$ by
\begin{align}
 \Module{+}{m} 
   &:= \left\{ f \in \ourE; H f = m f, X^{+} f = 0, (X^{-})^{m+1} f=0 \right\},
\notag
     \\[3pt]
 \Module{-}{m} 
   &:= \left\{ f \in \ourE; H f = -m f, X^{-} f = 0, (X^{+})^{m+1} f=0 \right\}.
\notag
\end{align}
Note that, by definition, each $f \in \Module{+}{m}$ (resp. $f \in \Module{-}{m}$ ) is a highest weight vector (resp. lowest weight vector) with respect to the $\g'=\mathfraksl_{2}$ action.
It is easy to see that $\Module{+}{m}$ is isomorphic to $\Module{-}{m}$ for any $m \in \N$, though $\Module{+}{0}=\Module{-}{0}$.
\end{definition}

In order to describe elements of $\Module{\pm}{m}$ explicitly, let us introduce a convergent power series $\Psi_{\alpha}$ on $\C$ for $\alpha \in \C \setminus (- \N )$ by
\[
	\Psi_{\alpha} (u) = \sum_{j=0}^{\infty} \frac{(-1)^{j}}{j!\,(\alpha)_{j}} u^{j}	\quad (u \in \C)
\]
with $(\alpha)_{j} = \Gamma(\alpha+j)/\Gamma(\alpha) = \alpha (\alpha+1) \cdots (\alpha+j-1)$.
Note that it is related to the Bessel function $J_{\nu}$ of the first kind. 
In fact, one has
\begin{equation}
\label{e:normalized_Bessel}
	\Psi_\alpha (u) = \Gamma(\alpha) u^{-(\alpha-1)/2} J_{\alpha-1}(2 u^{1/2}).
\end{equation}
Now we set 
\begin{equation}
\label{e:radial_part_psi}
	\psi_{\alpha}:=\Psi_{\alpha}(r_{x}^{2} r_{y}^{2}/4)
\end{equation}
for brevity.
We will find it convenient to write $\rho_{x}:=r_{x}^{2}/2,\,\rho_{y}:=r_{y}^{2}/2$.
%
Then a typical element $f \in \Module{+}{m}$ is given by
\begin{equation}
\label{e:typical1}
	  f = h_1(x) h_2(y) \rho_y^{\mu_{-}} \psi_{\kappa_+}
\end{equation}
with $h_1 \in \Harm^k(\R^p),\; h_2 \in \Harm^l(\R^q)$ and $\mu_{-}=\frac12(m + \kappa_{+} - \kappa_{-}) \in \N$,
while a typical one $f \in \Module{-}{m}$ is given by
\begin{equation}
\label{e:typical2}
	  f = h_1(x) h_2(y) \rho_x^{\mu_{+}} \psi_{\kappa_-}
\end{equation}
with $h_1 \in \Harm^k(\R^p)\; h_2 \in \Harm^l(\R^q)$ and $\mu_{+}=\frac12(m - \kappa_{+} + \kappa_{-}) \in \N$ (see \cite[Proposition 3.2]{min_rep}).
A general element of $\Module{+}{m}$ (resp. $\Module{-}{m}$) is a linear combination of these elements given in \eqref{e:typical1} (resp. \eqref{e:typical2}).

\begin{remark}
\label{r:radial_part}
All $K$-homogeneous elements of $\Module{+}{m}$ with $K$-type $(\kappa_{+},\kappa_{-})$ have the same radial part $\rho_{y}^{\mu_{-}} \psi_{\kappa_{+}} \in \C \llbracket r_{x}^{2},r_{y}^{2} \rrbracket$ with $\mu_{-}=\frac12 (m+\kappa_{+}-\kappa_{-})$.
Similarly, all $K$-homogeneous elements of $\Module{-}{m}$ with $K$-type $(\kappa_{+},\kappa_{-})$ have the same radial part $\rho_{x}^{\mu_{+}} \psi_{\kappa_{-}}$ with $\mu_{+}=\frac12 (m-\kappa_{+}+\kappa_{-})$.
\end{remark}

Now we collect a few facts from \cite{min_rep} that we need in order to prove our main result.

First, we recall the action of $\mathfrak p$ on $\Module{\pm}{m}$ (\cite[Lemma 4.1]{min_rep}).
Let $X_{i,p+j}$, $i \in [p],\, j \in [q]$, be a basis of $\mathfrak p$ given in \eqref{e:basis4g}.
Then, for $f=h_1 h_2 \rho_y^{\mu_{-}} \psi_{\kappa_+} \in \Module{+}{m}$ with $\kappa(f)=(\kappa_{+},\kappa_{-})$ and $\mu_{-}=\frac12(m+\kappa_{+}-\kappa_{-})$, one sees that
\begin{subequations}
\label{e:p-action_on_HWV}
\begin{equation}
\allowdisplaybreaks{
\begin{aligned}
  - \ai \pi (X_{i,p+j}) f
    = & \, \frac{\kappa_{-} + \mu_{-} - 1}{\kappa_{-}-1} (\pd_{x_i} h_1) (\pd_{y_j} h_2) \rho_y^{\mu_{-}} \psi_{\kappa_{+}-1}     \\
      & +  \mu_{-} (\pd_{x_i} h_1) (y_j h_2)^\dagger \rho_y^{\mu_{-}-1} \psi_{\kappa_{+}-1}     \\
      & +  \frac{\kappa_{+}-\kappa_{-}-\mu_{-}}{\kappa_{+}(\kappa_{-}-1)} (x_i h_1)^\dag (\pd_{y_j} h_2) \rho_y^{\mu_{-}+1} \psi_{\kappa_{+}+1}   \\
      & +  \frac{\kappa_{+}-\mu_{-}-1}{\kappa_{+}} (x_i h_1)^\dagger (y_j h_2)^\dagger \rho_y^{\mu_{-}} \psi_{\kappa_{+}+1},
\end{aligned}
}
\end{equation}
%
and, for $f=h_1 h_2 \rho_x^{\mu_{+}} \psi_{\kappa_-} \in \Module{-}{m}$ with $\kappa(f)=(\kappa_{+},\kappa_{-})$ and $\mu_{+}=\frac12(m-\kappa_{+}+\kappa_{-})$, 
\begin{equation}
\allowdisplaybreaks{
\begin{aligned}
   - \ai \pi (X_{i,p+j}) f
     = & \, \frac{\kappa_{+} + \mu_{+} - 1}{\kappa_{+}-1} (\pd_{x_i} h_1) (\pd_{y_j} h_2) \rho_x^{\mu_{+}} \psi_{\kappa_{-}-1}   \\
       & + \frac{\kappa_{-}-\kappa_{+}-\mu_{+}}{\kappa_{-}(\kappa_{+}-1)} (\pd_{x_i} h_1) (y_j h_2)^\dagger \rho_x^{\mu_{+}+1} \psi_{\kappa_{-}+1}   \\
       & + \mu_{+} (x_i h_1)^\dag (\pd_{y_j} h_2) \rho_x^{\mu_{+}-1} \psi_{\kappa_{-}-1}    \\
       & + \frac{\kappa_{-}-\mu_{+}-1}{\kappa_{-}} (x_i h_1)^\dagger (y_j h_2)^\dagger \rho_x^{\mu_{+}} \psi_{\kappa_{-}+1}.
\end{aligned}
}
\end{equation}
\end{subequations}
Here, in general, for a homogeneous polynomial $P=P(x_1,\dots,x_{\sfn})$ on $\R^{\sfn}$ of degree $d$, we set
\begin{equation}
\label{e:daggered_polynom}
  P^{\dag} := P - \frac{\rho}{2d + \sfn - 4} \Delta P,
\notag
\end{equation}
with $\Delta=\sum_{i=1}^{\sfn} \pd_{x_i}^2$ and $\rho=(1/2)\sum_{i=1}^{\sfn} x_i^2$.
Note that if $\Delta^2 P=0$ then $P^\dag$ is harmonic and that if $h=h(x_1,\dots, x_{\sfn})$ is harmonic then $\Delta(x_i h)=2 \pd_{x_i} h$ and $\Delta^2(x_i h)=0$.

Next, the following facts describe the $K$-types of $\Module{\pm}{m}$ and when they are irreducibile (\cite[Theorem 4.6]{min_rep}). 
\begin{facts}
\label{t:K-type formula}
Assume that $p \geqsl 1$, $q \geqsl 1$ and $p+q \in 2 \N$.
Let $m \in \N$ be a non-negative integer satisfying $m+3 \leqsl (p+q)/2$.
Then one has the following.
\begin{enumerate}
    \item
\label{thm_item:K-type}

     The $K$-type formula of $\Module{\pm}{m}$ is given by 
     \begin{equation}
	 \label{e:K-type_Module_m}
     \left. \Module{\pm}{m} \right|_{K} 
       \simeq \bigoplus_{\substack{k, l \geqsl 0 \\  k+\frac{p}2 - (l + \frac{q}2)  \in \Lambda_m}} \Harm^k(\R^p) \otimes \Harm^l(\R^q),
     \end{equation}  
     where $\Lambda_m=\{-m, -m+2, -m+4, \dots, m-2, m\}$.

    \item
\label{thm_item:irreducibility}
     Suppose further that $p, q \geqsl 2$.
     Then $\Module{\pm}{m}$ are irreducible $(\g,K)$-modules.
\end{enumerate}
\end{facts}

Note in particular that if the $K$-type of $f \in \Module{\pm}{m}$ is $\kappa(f)=(\kappa_{+},\kappa_{-})$, 
then it follows from \eqref{e:K-type_Module_m} that
\begin{equation}
\label{e:range_kappa+_kappa-}
	\abs{\kappa_{+} - \kappa_{-}} \leqsl m.
\end{equation}

Henceforth in the rest of the paper, we assume that
\begin{equation}
\label{e:assumpntion_on_pq}
	p \geqsl 2,\; q \geqsl 2,\quad p+q \in 2\N \quad \text{and}\quad  m+3 \leqsl \frac{p+q}2
\end{equation}
so that $\Module{\pm}{m}$ is irreducible.

%
%
\section{Annihilator of $\Module{\pm}{m}$}
\label{s:sec2}

In order to translate objects in terms of $\gnat$ into ones in terms of $\g$, and vice versa, let us fix an isomorphism $\Phi$ between $\g$ and $\gnat$ given by
\begin{equation}
	\Phi : \g \xrightarrow{\sim} \gnat, \quad X \mapsto I_{p,q}^{1/2}\, X\, I_{p,q}^{\,-1/2},
\end{equation}
where $I_{p,q}^{1/2}=\diag(1,\dots,1,\ai,\dots,\ai)$ and $I_{p,q}^{-1/2}$ is its inverse.

Let $\{M_{i,j}\}$ be generators of $\gnat$ given by
\begin{align}
\label{e:basis4gnat}
  M_{i,j} &= E_{i,j} - E_{j,i}
\end{align}
for $i,j=1,2,\dots,n$.
Then, one has
\begin{equation}
\label{e:Phi}
	\Phi(X_{i,j}) = \begin{cases}
						M_{i,j} 		&\text{if}\; i,j \in [p], \\
						-M_{i,j}  		&\text{if}\; i,j \in p+[q], \\
						\ai M_{i,j} 	&\text{otherwise}.
					\end{cases}
\end{equation}
We extend $\Phi$ to the isomorphism from $\bigoplus_{j=0}^{\infty} T^{j}(\g)$ onto $\bigoplus_{j=0}^{\infty} T^{j}(\gnat)$, where $T^{j}(\g):=\g \otimes \cdots \otimes \g$ ($j$ factors) etc..

Let $U(\g)$ denote the universal enveloping algebra of $\g$, and $U_j(\g)$ its subspace spanned by products of at most $j$ elements of $\g$, with $U_{-1}(\g)=0$.
It is well known that the associated graded algebra $\gr U(\g) = \bigoplus_{j=0}^{\infty} U_{j}(\g)/U_{j-1}(\g)$ is isomorphic to the symmetric algebra $S(\g)=\bigoplus_{j=0}^{\infty} S^{j}(\g)$.
Let $\sigma: \gr U(\g) \to S(\g)$ be the algebra isomorphism, and $\sigma_{j}$ its $j$-th piece:
\begin{equation}
	\sigma_{j}: U_{j}(\g)/U_{j-1}(\g) \xrightarrow{\sim} S^{j}(\g).  
\end{equation}
Let $\gamma: S(\g) \to U(\g)$ be the linear isomorphism called \emph{symmetrization}, and $\gamma_{j}$ its $j$-th piece
\begin{equation}
\label{e:gamma_j}
	\gamma_{j} : S^{j}(\g) \to U(\g),
\end{equation}
which is an injective linear map onto a vector-space complement to $U_{j-1}(\g)$ in $U_{j}(\g)$.
Then the following diagram is commutative:
\begin{center}
\begin{tikzcd}
								& U_{j}(\g) \ar[d,"p_j"] \\
	S^j(\g)	\ar[ur,"\gamma_j"',"1:1"]	& U_{j}(\g)/U_{j-1}(\g) \ar[l,"\sim"',"\sigma_j"],
\end{tikzcd}
\end{center}
where $p_j$ is the canonical projection.
In particular, $(\sigma_j \circ p_j \circ \gamma_j)(v) = v$ for all $v \in S^j(\g)$. 

In what follows, given a finite-dimensinal vector space $U$ over $\C$,
we will identify $S^j (U)$ (resp. $\mwedge{j} U$) with the subspace of symmetric (resp. alternating) tensors in $U^{\otimes j}$.

For $u \in U(\g)$, we sometimes denote $\pi(u)$ by $\hat u$ for brevity.
Let $\Omega_{\g},\, \Omega_{\mathfrak o_p}$ and $\Omega_{\mathfrak o_q}$ be the Casimir elements of $\g, \, \mathfrak o_p$ and $\mathfrak o_q$ respectively%
	\footnote{We regard $\mathfrak o_{p}$ and $\mathfrak o_{q}$ as subalgebras of $\mathfrak k \simeq \mathfrak o_{p} \oplus \mathfrak o_{q}$ canonically.}%
, which are given by 
\begin{equation}
\label{e:casimir_elements}
	\Omega_{\g} = \sum_{\substack{i,j \in [n] \\ i<j}} X_{i,j} X_{i,j}^{\vee},
	\quad
	\Omega_{\mathfrak o_p} = \sum_{\substack{i,j \in [p] \\ i<j}} X_{i,j} X_{j,i},
	\quad
	\Omega_{\mathfrak o_q} = \sum_{\substack{i,j \in [q] \\ i<j}} X_{i,j} X_{j,i}.
\end{equation}
Then the corresponding operators represented in $\PD(V)$ are given by
\begin{subequations}
\label{e:Casimirs}
\begin{align}
	\hat \Omega_{\g}	&= (\Eu_x-\Eu_y)^2 + (p-q) (\Eu_x-\Eu_y)  - 2 (\Eu_x + \Eu_y)  
		\notag \\ 
			& \hspace{2em} - \left( r_x^2 \, r_y^2 + r_x^2 \, \Delta_x + r_y^2 \, \Delta_y + \Delta_x \, \Delta_y \right) - p q, 
		\label{e:Casimir_g} \\
	\hat \Omega_{\mathfrak o_{p}}	&= \Eu_x^2 + (p-2) \Eu_x -r_x^2 \, \Delta_x,  
		\label{e:Casimir_op}	\\
	\hat \Omega_{\mathfrak o_{q}}	&=\Eu_y^2 + (q-2) \Eu_y - r_y^2 \, \Delta_y.
		\label{e:Casimir_oq}
\end{align}
\end{subequations}
It is well known that the Casimir operator $\hat \Omega _{\g'}$ of $\g'=\mathfrak{sl}_{2}$ given by 
\[
	\hat \Omega _{\g'} = H^{2} + 2(X^{+} X^{-} + X^{+} X^{-})
\]
satisfies the relation 
\begin{equation*}
\label{e:relation_2casimirs}
  \hat \Omega_{\g} = \hat \Omega_{\g'} - \frac14 (p+q)^2 + (p+q)
\end{equation*}
(see \cite{Howe79, RS80}),
which implies that
\begin{equation}
\label{e:eigenvalue_casimir}
	\hat \Omega_{\g} |_{\Module{\pm}{m}} = m(m+2) - \frac14 (p+q)^{2}+ (p+q)
\end{equation}
since $\hat \Omega_{\g'}$ acts on ${\Module{\pm}{m}}$ by the scalar $m(m+2)$.

Let $\alpha$ denote the highest root of $\g$ relative to any positive root system.
Then $S^{2}(\g)$ decomposes as $S^{2}(\g) \simeq V_{2 \alpha} \oplus W$, 
where $V_{2 \alpha}$ is the irreducible $\g$-module with highest weight $2 \alpha$ and $W$ is the $\g$-invariant complement of $V_{2 \alpha}$ (see \eqref{e:irrep_decomp_S2(o_n)} below).

%
%
\begin{theorem}[Garfinkle]
\label{t:garfinkle}
Let $\mathscr I$ be an ideal of infinite codimension in $U(\g)$ and set $I=\bigoplus_{j=0}^{\infty} (\mathscr I \cap U_{j}(\g)) / (\mathscr I \cap U_{j-1}(\g)) \subset \gr U(\g)$.
Then $\mathscr I$ is the Joseph ideal if and only if $\sigma(I) \cap S^2(\g)=W$.
\end{theorem}
In fact, it is shown that Theorem \ref{t:garfinkle} holds true for all complex simple Lie algebras not of type $A$ in \cite{Garfinkle82,GS04}.

Using the notation in Appendix \ref{s:app_irrep_S2} below, the $\g$-invariant complement $W$ is given by 
\begin{equation}
\label{e:W}
	W = {E}^{\g}_{\varnothing} \oplus {E}^{\g}_{\ydiagram{1,1,1,1}} \oplus {E}^{\g}_{\ydiagram{2}},
\end{equation} 
where 
\begin{equation}
	E^{\g}_{\lambda} = \Phi^{-1}(\nat{E}_{\lambda}) \qquad (\lambda \in \{\varnothing \,, \ydiagram{1,1,1,1} \,, \ydiagram{2}\}).
\end{equation}


The following lemma is trivial, and its proof is left to the reader.
\begin{lemma}
\label{l:natQ_and_Q}
Define an element $Q \in T^{2}(\g)$ by $Q=\sum_{i,j} X_{i,j} \otimes X_{i,j}^{\vee}$. 
Then $Q$ is a symmetric tensor and satisfies $Q=\Phi^{-1}(\nat{Q})$ and $\gamma_{2}(Q)=2 \Omega_{\g}$.
\end{lemma}
%


%
%
%
%
Following \cite{BZ91}, let us introduce an element $\nat{\Xi} \in S^{2}(\gnat)$, or $\Xi \in S^{2}(\g)$, that plays a critical role in the proof of our main result:
\begin{equation}
	\nat{\Xi} := \frac12 \sum_{i,j} \eta^{ij} \nat{S}_{ij} = \frac12 \Big( \sum_{i \in [p]} \nat{S}_{ii} - \sum_{i \in p+[q]} \nat{S}_{ii} \Big),
\end{equation} 
where $\eta^{ij}$ denotes the $(i,j)$-th entry of $I_{p,q}^{-1}=I_{p,q}$.

\begin{lemma}
\label{l:Xi_and_Omegas}
Set $\Xi:=\Phi^{-1}(\nat{\Xi}) \in S^{2}(\g)$.
Then it satisfies that
\begin{equation}
	\gamma_{2} (\Xi) = \Omega_{\mathfrak o_{p}} - \Omega_{\mathfrak o_{q}} - \frac{p-q}{p+q} \Omega_{\g}.
\end{equation}
\end{lemma}
\begin{proof}
Since $\nat{S}_{ii}=\sum_{k} M_{i,k} \otimes M_{k,i} -\frac1n \nat{Q}$, one sees
\begin{align*}
	\nat{\Xi} &= \frac12 \sum_{k} \Big( \sum_{i \in [p]} - \sum_{i \in p+[q]} \Big) M_{i,k} \otimes M_{k,i} - \frac{p-q}{2n} \nat{Q}
	\\
		&= \frac12 \Big( \sum_{i,k \in [p]} + \sum_{\substack{i \in [p] \\ k \in p+[q]}} \Big) M_{i,k} \otimes M_{k,i} 
				- \frac12 \Big( \sum_{\substack{i \in p+[q] \\ k \in [p]}} + \sum_{i,k \in p+[q]} \Big) M_{i,k} \otimes M_{k,i} - \frac{p-q}{2n} \nat{Q} 
	\\
		&= \frac12 \Big( \sum_{i,k \in [p]} - \sum_{i,k \in p+[q]} \Big) M_{i,k} \otimes M_{k,i} - \frac{p-q}{2n} \nat{Q}.
\end{align*}
Therefore, it follows from \eqref{e:Phi} that
\begin{align}
	\Xi &= \frac12 \Big( \sum_{i,k \in [p]} X_{i,k} \otimes X_{k,i} - \sum_{i,k \in p+[q]} (-X_{i,k}) \otimes (-X_{k,i}) \Big) - \frac{p-q}{2n} \Phi^{-1}(\nat{Q})
		\notag \\
		&= \sum_{\substack{i,k \in [p] \\ i<k}} X_{i,k} \otimes X_{k,i} - \sum_{\substack{i,k \in p+[q] \\ i<k}} X_{i,k} \otimes X_{k,i} - \frac{p-q}{2n} Q. 
\end{align}
Now, the result immediately follows from Lemma \ref{l:natQ_and_Q}.
\end{proof}

\begin{proposition}
\label{p:eigenvalues_casimir_k}
For a $K$-homogeneous $f \in \Module{\pm}{m}$ with $\kappa(f)=(\kappa_{+},\kappa_{-})$, one has
\begin{align}
	\hat \Omega_{\mathfrak o_{p}} f &= \Big( (\kappa_{+}-1)^{2} -\frac{(p-2)^{2}}4 \Big) f,
\label{e:eigenvalues_casimir_k1}
		\\
	\hat \Omega_{\mathfrak o_{q}} f &= \Big( (\kappa_{-}-1)^{2} -\frac{(q-2)^{2}}4 \Big) f.	
\label{e:eigenvalues_casimir_k2}
\end{align}
\end{proposition}
\begin{proof}
We only show the case of $\Module{+}{m}$ here.
The other case can be proved exactly in the same manner.

First, we show \eqref{e:eigenvalues_casimir_k1}.
It suffices to prove it for a typical element $f=h_{1} h_{2} \rho_{y}^{\mu_{-}} \psi_{\kappa_{+}} \in \Module{+}{m}$ with $h_{1} \in \Harm^{k}(\R^{p}),\,h_{2} \in \Harm^{l}(\R^{q})$ and $\mu_{-}=(1/2)(m+\kappa_{+}-\kappa_{-})$, where $(\kappa_{+},\kappa_{-})=(k+\frac{p}2,l+\frac{q}2)$.
Recall that for a homogeneous polynomial $h=h(x_1,\dots,x_{\sfn})$ on $\R^{\sfn}$ of degree $d$ and for a smooth function $\phi(u)$ in a single variable $u$, we have
\begin{equation}
\label{e:laplacian_applied_to_product2}
	\Delta ( h \phi(\rho) ) = \left( 2 d + n \right) h \phi'(\rho) + 2 h \rho \phi''(\rho),
\end{equation}
where $\Delta=\sum_{i=1}^{\sfn} \pd_{x_i}^2$ and $\rho=(1/2)\sum_{i=1}^{\sfn} x_i^2$ (cf.~\cite[Lemma 3.2]{min_rep}).
Thus, for $f=h_{1} h_{2} \rho_{y}^{\mu_{-}} \psi_{\kappa_{+}}$, we have
\allowdisplaybreaks{
\begin{align*}
	\Delta_{x} f &= h_{2} \rho_{y}^{\mu_{-}} \Delta_{x}(h_{1} \Psi_{\kappa_{+}}(\rho_{x} \rho_{y})) 
		\\
		&= (2k+p) h_{1} h_{2} \rho_{y}^{\mu_{-}+1} \psi'_{\kappa_{+}} + 2 h_{1} h_{2} \rho_{x} \rho_{y}^{\mu_{-}+2} \psi''_{\kappa+},
\intertext{%
where we set $\psi'_{\kappa_{+}}=\Psi'(\rho_{x} \rho_{y}),\,\psi''_{\kappa_{+}}=\Psi''_{\kappa_{+}}(\rho_{x} \rho_{y})$.
Similarly, one sees that
}
	E_{x} f		&= k h_1 h_2 \rho_y^{\mu_-} \psi_{\kappa_+} + 2 h_1 h_2 \rho_x \rho_y^{\mu_- + 1} \psi'_{\kappa_+},
		\\
	E_{x}^{2} f	&= k^2 h_{1} h_{2} \rho_{y}^{\mu_{-}} \psi_{\kappa_{+}} + 4(k+1) h_{1} h_{2} \rho_{x} \rho_{y}^{\mu_{-}+1} \psi'_{\kappa_+} 
		\\
		&\hspace{1em} + 4 h_{1} h_{2} \rho_{x}^{2} \rho_{y}^{\mu_{-}+2} \psi''_{\kappa_{+}}.
\intertext{Therefore,}
	\hat\Omega_{\mathfrak o_{p}} f	&= \big( E_{x}^{2} + (p-2) E_{x} - 2 \rho_{x} \Delta_{x} \big) \, h_{1} h_{2} \rho_{y}^{\mu_{-}} \psi_{\kappa_{+}}
		\\
		&= \big( k^{2}+(p-2)k \big) \, h_{1} h_{2} \rho_{y}^{\mu_{-}} \psi_{\kappa_{+}} 
		\\
		&= \Big( \big( k+\frac{p-2}2 \big)^{2} - \frac{(p-2)^{2}}4 \Big) f.
\end{align*}
}
Now, substituting $\kappa_{+}=k+p/2$, one obtains \eqref{e:eigenvalues_casimir_k1}.

Next, we turn to \eqref{e:eigenvalues_casimir_k2}.
Similar calculations yield
\allowdisplaybreaks{
\begin{align*}
	\Delta_{y} f &= h_{1} \Delta_{x}( h_{2} \rho_{y}^{\mu_{-}} \Psi_{\kappa_{+}}(\rho_{x} \rho_{y})) 
		\\
		&= \big( (2l+q)\mu_{-} +2\mu_{-}(\mu_{-}-1) \big) h_{1} h_{2} \rho_{y}^{\mu_{-}-1} \psi_{\kappa_{+}} \\
		&\hspace{1em} + \big( (2l+q)+4\mu_{-} \big) h_{1} h_{2} \rho_{x} \rho_{y}^{\mu_-} \psi'_{\kappa_+} 
					+ 2 h_{1} h_{2} \rho_{x}^{2} \rho_{y}^{\mu_{-}+1} \psi''_{\kappa+},
		\\
	E_{y} f		&= (l+2\mu_{-}) h_1 h_2 \rho_y^{\mu_-} \psi_{\kappa_+} + 2 h_1 h_2 \rho_x \rho_y^{\mu_- + 1} \psi'_{\kappa_+},
		\\
	E_{y}^{2} f	&= (l+2\mu_{-})^2 h_{1} h_{2} \rho_{y}^{\mu_{-}} \psi_{\kappa_{+}} 
					+ 4(l+2\mu_{-}+1) h_{1} h_{2} \rho_{x} \rho_{y}^{\mu_{-}+1} \psi'_{\kappa_+} \\
		&\hspace{1em} + 4 h_{1} h_{2} \rho_{x}^{2} \rho_{y}^{\mu_{-}+2} \psi''_{\kappa_{+}},
\intertext{and}
	\hat\Omega_{\mathfrak o_{q}} f &=\big( (l+2\mu_{-})^{2}+(q-2)(l+2\mu_{-})-2\mu_{-}(2l+q+2\mu_{-}-2)\big) f
\end{align*}
}
for $f=h_{1}h_{2}\rho_{y}^{\mu_{-}}\psi_{\kappa_{+}}$.
Substituting $\kappa_{-}=l+q/2,\,\mu_{-}=(1/2)(m+\kappa_{+}-\kappa_{-})$, one obtains \eqref{e:eigenvalues_casimir_k2}. 
\end{proof}

\begin{corollary}
\label{p:key_lemma}
Set $\hat \Xi:= \pi (\gamma_{2} (\Xi) ) \in \PD(V)$. 
Then, for a $K$-homogeneous $f \in \Module{\pm}{m}$ with $\kappa(f)=(\kappa_{+},\kappa_{-})$, one has $\hat \Xi f=\lambda_{\kappa(f)} f$, where $\lambda_{\kappa(f)}$ is a scalar given by
\begin{equation}
\label{e:key_lemma}
	\lambda_{\kappa(f)} = (\kappa_{+} - \kappa_{-})(\kappa_{+} + \kappa_{-} - 2) - \frac{p-q}{p+q} m(m+2).
\end{equation}
In particular, $\lambda_{\kappa(f)}=0$ if $m=0$.
\end{corollary}
\begin{proof}
The first statement immediately follows from Lemma \ref{l:Xi_and_Omegas} and Proposition \ref{p:eigenvalues_casimir_k}.
As for the second, we note that if $m=0$ then $\kappa_{+}=\kappa_{-}$ by \eqref{e:range_kappa+_kappa-}.
%
\end{proof}

In what follows, let $\Anni{\pm}{m}$ denote the annihilator in $U(\g)$ of our $(\g,K)$-module $\Module{\pm}{m}$ for short: 
\[
	\Anni{\pm}{m}:=\Ann_{U(\g)} \Module{\pm}{m}.
\]
Correspondingly, set $I^{\pm}(m):=\bigoplus_{j=0}^{\infty} (\Anni{\pm}{m} \cap U_{j}(\g))/(\Anni{\pm}{m} \cap U_{j-1}(\g))$. 

%
%

\begin{theorem}
\label{t:main result}
The annihilator $\Anni{\pm}{m}$ is the Joseph ideal if ond only if $m=0$.
\end{theorem}

\begin{proof}
Recall that our $\g$-invariant complement $W \subset S^{2}(\g)$ is given by
\[
	W = {E}^{\g}_{\varnothing} \oplus {E}^{\g}_{\ydiagram{1,1,1,1}} \oplus {E}^{\g}_{\ydiagram{2}}.
\]
We must show that $W = \sigma(I^{\pm}(m)) \cap S^{2}(\g)$ if and only if $m=0$ by Theorem \ref{t:garfinkle}.
However, the $\g$-invariance of both sides reduces the problem to showing that $W \subset \sigma(I^{\pm}(m)) \cap S^{2}(\g)$ if and only if $m=0$. 

Since $\Omega_{\g} \in U_{2}(\g)$ acts on $\Module{\pm}{m}$ by the scalar given by \eqref{e:eigenvalue_casimir}, say $\lambda_{0}$, 
\[
	Q=\sigma_{2}(p_{2}( 2(\Omega_{\g} - \lambda_{0})) ) \in \sigma(I^{\pm}(m)),
\]
and hence we have
\begin{equation}
	{E}^{\g}_{\varnothing} \subset \sigma(I^{\pm}(m)) \cap S^{2}(\g).
\end{equation}

Next, a simple calculation shows that $\pi(\Phi^{-1}(\nat{S}_{ijkl}))=0$ in $\PD(V)$ for all $i<j<k<l$.
Namely, we have 
\begin{equation}
	{E}^{\g}_{\ydiagram{1,1,1,1}} \subset \sigma(I^{\pm}(m)) \cap S^{2}(\g).
\end{equation}

Finally, we show that
\begin{equation}
\label{e:criterion4E_(2)}
	{E}^{\g}_{\ydiagram{2}} \subset \sigma(I^{\pm}(m)) \cap S^{2}(\g) \quad \text{if and only if} \quad m=0.
\end{equation}
Since $E^{\g}_{\ydiagram{2}}$ is irreducible, in order that \eqref{e:criterion4E_(2)} holds, it suffices to show
\begin{equation}
\label{e:criterion4E_(2)_2}
	\Xi \in \sigma(I^{\pm}(m)) \cap S^{2}(\g) \quad \text{if and only if} \quad m=0.
\end{equation}
Assume that $\Xi$ is in $\sigma(I^{\pm}(m)) \cap S^{2}(\g)$.
Then there exists an element $u \in \Anni{\pm}{m} \cap U_{2}(\g)$ satisfying $\Xi=\sigma_{2}(p_2(u))$ and 
\begin{equation}
\label{e:property_u}
	\pi(u) f = 0 \quad \text{for all}\; f \in \Module{\pm}{m}.
\end{equation}
Since $\Xi = \sigma_2(p_2(\gamma_2(\Xi)))$, there exists an element $u_1 \in U_1(\g)$ such that $\gamma_2(\Xi)=u+u_1$.
Moreover, $u_1$ can be written as $u_1=\lambda+Y$ with $\lambda \in \C$ and $Y \in \g$ since $U_1(\g)=\C \oplus \g$.
Now, it follows from \eqref{e:property_u} and Corollary \ref{p:key_lemma} that
\begin{equation}
\label{e:key_relation}
\begin{aligned}
	\pi(Y) f &= (\lambda_{\kappa(f)}-\lambda)f 
\\
		&=\Big( (\kappa_{+}-\kappa_{-})(\kappa_{+}+\kappa_{-}+2) -\frac{p-q}{p+q}m(m+2) - \lambda \Big) f
\end{aligned}
\end{equation}
for all $K$-homogeneous $f \in \Module{\pm}{m}$ with $\kappa(f)=(\kappa_{+},\kappa_{-})$.
In view of the right-hand side of \eqref{e:key_relation}, the $\mathfrak p$-part action given in \eqref{e:p-action_on_HWV} forces this $Y$ to be in $\mathfrak k$.
Thus, letting $Y=Y_{1}+Y_{2}$, $Y_{1} \in \mathfrak{o}_{p},\,Y_{2} \in \mathfrak{o}_{q}$, be the decomposition corresponding to $\mathfrak k=\mathfrak o_{p} \oplus \mathfrak o_{q}$, one sees that \eqref{e:key_relation} is equivalent to the fact that 
$\pi(Y_{1})$ acts on $\Harm^{k}(\R^{p})$ by a scalar and $\pi(Y_{2})$ acts on $\Harm^{l}(\R^{q})$ by a scalar, 
where $\kappa_{+}=k+\frac{p}2$ and $\kappa_{-}=l+\frac{q}2$.
If $m=0$ then $\Xi \in \sigma(I^{\pm}(m))$; in fact, one can take $Y=0$ and $\lambda=0$ since $\lambda_{\kappa(f)}=0$.
On the other hand, if $m \ne 0$, this is impossible, and hence $\Xi \notin \sigma(I^{\pm}(m))$. 
Therefore, \eqref{e:criterion4E_(2)_2} holds.
\end{proof}

%
%

\appendix

\section{Irreducible decomposition of $S^{2}(\mathfrak o_{n})$}
\label{s:app_irrep_S2}

Let $E$ denote the natural representation of $\gnat$, i.e., $E=\C^{n}$ equipped with the standard bilinear form $\beta_{0}$, where $\beta_{0}(u,v)=\tp{u}v$ ($u,v \in E$).
Then, 
\begin{equation*}
	\varphi: \mwedge{2} E \to \gnat, \; v \wedge w \mapsto -\iota(\cdot) (v \wedge w)
\end{equation*}
provides an isomorphism between $\gnat$-modules, where $\iota$ denotes the contraction with respect to $\beta_{0}$.
Hence, we have $S^{2}(\mwedge{2} E) \simeq S^{2}(\mathfrak o_{n})$.

It is well known that any irreducible representation of $\gnat$ is realized in the subspace of traceless tensors of $E^{\otimes k}$ for some $k \in \N$ and is parameterized by Young diagram.
Noting that $\mwedge{2} E = E_{\ydiagram{1,1}}$, it follows that $\mwedge{2} E \otimes \mwedge{2} E$ decomposes as 
\begin{equation}
\label{e:irr_decomp_A2E_and_A2E}
	\mwedge{2} E \otimes \mwedge{2} E = E_{\varnothing} \oplus E_{\ydiagram{1,1,1,1}} \oplus E_{\ydiagram{2}} \oplus E_{\ydiagram{2,2}} \oplus E_{\ydiagram{2,1,1}} \oplus E_{\ydiagram{1,1}},
\end{equation}
which can be verified by calculating Littlewood-Richardson coefficients in terms of universal characters (see e.g.~\cite{OkadaBook06_2}).
Note also that $S^{2}(\mwedge{2} E)$ is equal to the direct sum of the first four factors in \eqref{e:irr_decomp_A2E_and_A2E}:
\begin{equation}
\label{e:irr_summand}
	E_{\varnothing}, 			\quad
	E_{\ydiagram{1,1,1,1}},		\quad
	E_{\ydiagram{2}},			\quad
	E_{\ydiagram{2,2}}.
\end{equation}
One can easy verify that $E_{\ydiagram{2,2}}$ is the highest weight module with highest weight $2 \alpha$, where $\alpha$ is the highest root of $\gnat$.

Next, let us describe the first three irreducible summands in \eqref{e:irr_summand} concretely.
Let $\{e_{i};i=1,\dots,n\}$ denote the standard basis of $\C^{n}$, i.e., $e_{i}=\tp{(0,\dots,\overset{i\text{-th}}{1},\dots,0)}$.
Then, one has
\begin{equation}
\begin{aligned}
	E_{\varnothing} &=\Big\< \sum_{i=1}^{n} {e_i}^{2} \Big\>_{\C} \,,
\\
	E_{\ydiagram{1,1,1,1}} & =\big\< e_i \wedge e_j \wedge e_k \wedge e_l ; 1 \leqsl i<j<k<l \leqsl n \big\>_{\C} \,,
\\
	E_{\ydiagram{2}} &= \Big\< e_i e_j-\frac1n \delta_{i,j} \sum_{k=1}^n {e_k}^2; i,j \in [n] \Big\>_{\C} \,.
\end{aligned}
\end{equation}
Thus, if one defines injective $\rmO_{n}$-equivariant linear maps by 
\begin{equation}
\begin{aligned}
	\varphi_{\varnothing} 			&: E_{\varnothing} \to S^2(\gnat), &
	& \quad \sum_{i=1}^{n} {e_i}^{2} \mapsto \nat{Q},
\\
	\varphi_{\ydiagram{1,1,1,1}} 	&: E_{\ydiagram{1,1,1,1}} \to S^2(\gnat), &
	& \quad e_{i} \wedge e_{j} \wedge e_{k} \wedge e_{l} \mapsto \nat{S}_{i j k l},
\\
	\varphi_{\ydiagram{2}} 		&: E_{\ydiagram{2}} \to S^2(\gnat), &
	& \quad e_{i} e_{j}-\frac1n \delta_{i j} \sum_{k=1}^{n} {e_k}^{2} \mapsto \nat{S}_{i j},
\end{aligned}
\end{equation}
with
\begin{align}
	\nat{Q} &=\sum_{i,k \in [n]}  M_{i,k} \otimes M_{k,i} = - \sum_{i,k \in [n]}  M_{i,k} \otimes M_{i,k},
\\
	\nat{S}_{i j k l} 
			&= \frac12 \big( M_{i,j} \otimes M_{k,l} - M_{i,k} \otimes M_{j,l} + M_{i,l} \otimes M_{j,k} \notag \\ 
			&\hspace{2em} + M_{k,l} \otimes M_{i,j} - M_{j,l} \otimes M_{i,k} + M_{j,k} \otimes M_{i,l} \big),
\\
	\nat{S}_{i j} 
			&=\frac12 \sum_{k \in [n]} ( M_{i,k} \otimes M_{k,j} + M_{k,j} \otimes M_{i,k} ) - \frac1{n} \delta_{i,j} \nat{Q},
\end{align}
then one sees that the irreducible decomposition of $S^{2}(\gnat)$ is given by
\begin{equation}
\label{e:irrep_decomp_S2(o_n)}
	S^{2}(\gnat) = \nat{E}_{\varnothing} \oplus \nat{E}_{\ydiagram{1,1,1,1}} \oplus \nat{E}_{\ydiagram{2}} \oplus \nat{E}_{\ydiagram{2,2}},
\end{equation}
where
\begin{equation}
\begin{aligned}
	\nat{E}_{\varnothing} &=\< \nat{Q} \>_{\C},
\\
	\nat{E}_{\ydiagram{1,1,1,1}} & =\< \nat{S}_{ijkl} ; 1 \leqsl i<j<k<l \leqsl n \>_{\C}, 
\\
	\nat{E}_{\ydiagram{2}} &= \< \nat{S}_{ij}; i,j \in [n] \>_{\C},
\end{aligned}
\end{equation}
and $\nat{E}_{\ydiagram{2,2}}$ is the image of $E_{\ydiagram{2,2}}$ under the isomorphism $S^{2}(\mwedge{2} E) \simeq S^{2}(\mathfrak o_{n})$.

\bibliographystyle{amsplain}

\providecommand{\bysame}{\leavevmode\hbox to3em{\hrulefill}\thinspace}
\providecommand{\MR}{\relax\ifhmode\unskip\space\fi MR }
\providecommand{\MRhref}[2]{%
  \href{http://www.ams.org/mathscinet-getitem?mr=#1}{#2}
}
\providecommand{\href}[2]{#2}

\end{document}

%% file: env.tex
\numberwithin{equation}{section}	

\theoremstyle{plain}

\newtheorem{theorem}{Theorem}[section]      
\newtheorem*{theorem*}{Theorem}

\newtheorem{lemma}[theorem]{Lemma}

\newtheorem{proposition}[theorem]{Proposition}

\newtheorem{corollary}[theorem]{Corollary}

\newtheorem*{conjecture*}{Conjecture}

\newtheorem{facts}[theorem]{Facts}
\newtheorem*{facts*}{Facts}

\theoremstyle{definition}
\newtheorem{definition}[theorem]{Definition}

\theoremstyle{remark}
\newtheorem{remark}[theorem]{Remark}



%% file: ncpfaff_def.tex
\def\geqsl{\geqslant}
\def\leqsl{\leqslant}

\renewcommand{\setminus}{\smallsetminus}  

\newcommand{\trace}[1]{\ensuremath{\operatorname{tr}{\left( #1 \right)}}}

\newcommand{\tp}[1]{\ensuremath{\hspace{1.8truept}\vphantom{1}^{t}\hspace{-1pt}#1}}

\newcommand{\diag}{\ensuremath{\operatorname{diag}}}

\newcommand{\N}{\ensuremath{\mathbb N}}

\newcommand{\R}{\ensuremath{\mathbb R}}
\newcommand{\C}{\ensuremath{\mathbb C}}

\newcommand{\gl}{\ensuremath{\mathfrak {gl}}}
\newcommand{\mathfraksl}{\ensuremath{\mathfrak {sl}}}

\newcommand{\g}{\ensuremath{\mathfrak{g}}}   

\newcommand{\GL}[1]{\ensuremath{\mathrm{GL}_{#1}}}  
\newcommand{\SL}[1]{\ensuremath{\mathrm{SL}_{#1}}}  
\newcommand{\Sp}{\ensuremath{\mathrm{Sp}}}          

\newcommand{\rmd}{\ensuremath{\operatorname{d}\!}}
\newcommand{\pd}{\ensuremath{\partial}}

\def\<{\langle}
\def\>{\rangle}

\def\c2vec#1#2{ %
   \left[ \begin{smallmatrix} %
           #1 \\ #2  \end{smallmatrix} %
   \right]}
